\documentclass[12pt]{amsart}
\usepackage{amsmath, amsthm, amssymb}
\usepackage[top=1.25in, bottom=1.25in, left=1.0in, right=1.0in]{geometry}

\allowdisplaybreaks
\usepackage{tkz-graph}
\usetikzlibrary{arrows}
\usetikzlibrary{shapes}
\usepackage[position=bottom]{subfig}

\makeatletter
\newtheorem*{rep@theorem}{\rep@title}
\newcommand{\newreptheorem}[2]{
\newenvironment{rep#1}[1]{
 \def\rep@title{#2 \ref{##1}}
 \begin{rep@theorem}}
 {\end{rep@theorem}}}
\makeatother

\theoremstyle{plain}
\newtheorem{thm}{Theorem}[section]
\newreptheorem{thm}{Theorem}

\newreptheorem{prop}{Proposition}
\newtheorem{lem}[thm]{Lemma}
\newreptheorem{lem}{Lemma}

\newreptheorem{conjecture}{Conjecture}
\newtheorem{cor}[thm]{Corollary}
\newreptheorem{cor}{Corollary}

\theoremstyle{definition}

\theoremstyle{remark}

\newtheorem*{FixerMove}{\bf {Fixer's turn}}
\newtheorem*{BreakerMove}{\bf {Breaker's turn}}

\newcommand{\fancy}[1]{\mathcal{#1}}

\newcommand{\IN}{\mathbb{N}}

\newcommand{\inj}{\hookrightarrow}

\newcommand{\set}[1]{\left\{ #1 \right\}}
\newcommand{\setb}[3]{\left\{ #1 \in #2 \mid #3 \right\}}
\newcommand{\setbs}[2]{\left\{ #1 \mid #2 \right\}}
\newcommand{\card}[1]{\left|#1\right|}

\newcommand{\floor}[1]{\left\lfloor#1\right\rfloor}
\newcommand{\func}[3]{#1\colon #2 \rightarrow #3}
\newcommand{\funcinj}[3]{#1\colon #2 \inj #3}

\newcommand{\irange}[1]{\left[#1\right]}

\newcommand{\parens}[1]{\left( #1 \right)}

\newcommand{\DefinedAs}{\mathrel{\mathop:}=}
\newcommand{\im}{\operatorname{im}}

\renewcommand{\S}{\fancy{S}}
\newcommand{\W}{\fancy{W}}

\renewcommand{\P}{\fancy{P}}

\newcommand{\F}{\mathfrak{F}}
\newcommand{\B}{\mathfrak{B}}

\title{A game generalizing Hall's theorem}
\author{Landon Rabern}

\begin{document}
\begin{abstract}
We characterize the initial positions from which the first player has a
winning strategy in a certain two-player game.  This provides a generalization
of Hall's theorem.  Vizing's edge coloring theorem follows from a special case.
\end{abstract}
\maketitle

\section{Introduction}
A $\emph{set system}$ is a finite family of finite sets. A \emph{transversal} of a set system $\S$ is an injection $\funcinj{f}{\S}{\bigcup \S}$ such that $f(S) \in S$ for each $S \in \S$.  Hall's theorem \cite{hall} gives the precise conditions under which a set system has a transversal.

\begin{thm}[Hall \cite{hall}]
A set system $\S$ has a transversal iff $\card{\bigcup \W} \geq \card{\W}$ for each $\W
\subseteq \S$.
\end{thm}

We generalize this by analyzing winning strategies in a two-player game played on a set system  by \emph{Fixer} (henceforth dubbed $\F$) and \emph{Breaker}. Fixer wins this game iff he can modify the set system so that it has a transversal; otherwise, Breaker wins.   Additionally, when playing on the set system $\S$, we require a \emph{pot} $P$ with $\bigcup \S \subseteq P$. The first move is always $\F$'s and he can do the following.

\begin{FixerMove}
Pick $x \in P$ and $S \in \S$ with $x \not \in S$ and replace $S$ with $S
\cup \set{x} \smallsetminus \set{y}$ for some $y \in S$.
\end{FixerMove}

We need the notation $\irange{k} \DefinedAs \set{1, \ldots, k}$. For each $t \in \irange{\card{\S} - 1}$, we have a different rule for Breaker, let's refer to Breaker by $\B_t$ when he is playing with the following rule
for $t$.  

\begin{BreakerMove}
If $\F$ modified $S \in \S$ by inserting $x$ and removing $y$, $\B_t$ can pick
at most $t$ sets in $\S \smallsetminus \set{S}$ and modify them by
swapping $x$ for $y$ or $y$ for $x$.
\end{BreakerMove}

To state the main theorem we need a couple more pieces of notation.  Define the
\emph{degree} in $\W \subseteq \S$ of $x \in P$ as
\[d_{\W}(x) \DefinedAs \card{\setb{S}{\W}{x \in S}}.\]

\noindent Now define the \emph{$t$-value} of $\W \subseteq \S$ as
\[\nu_t(\W) \DefinedAs \sum_{x \in \bigcup \W} \floor{\frac{d_{\W}(x) -
1}{t + 1}}.\]

Intuitively, this measures how many new elements from the pot $\F$ can swap into the sets without $\B_t$ undoing the progress by effectively interchanging the names of the swapped elements. For instance, if $d_{\W}(x) \leq t + 1$ and $\F$ swaps $y$ in for $x$, then $\B_t$ can swap $y$ in for all the other $x$'s since there are at most $t$ of them.  In this case $x$ contributes nothing to the $t$-value of $\W$.  Our main theorem shows that this intuition is correct.

\begin{thm}\label{MainTheorem}
In a set system $\S$ with $\bigcup \S \subseteq P$ and $\card{P} \ge \card{\S}$, $\F$ has a winning strategy against $\B_t$ iff $\card{\bigcup \W} \geq \card{\W} - \nu_t(\W)$ for each $\W \subseteq \S$.
\end{thm}

We can recover Hall's theorem from the $t = \card{\S} - 1$ case; to wit: $\B_t$
can remove all $y$'s in $\S$ rendering $\F$'s move equivalent to swapping the
names of $x$ and $y$, that is, rendering it useless. In Section
\ref{VizingSection} we show that Vizing's edge coloring theorem is a quick corollary of this result.  
In fact, the strategy employed by $\F$ is based, in part, on Ehrenfeucht, Faber and Kierstead's proof of Vizing's theorem
\cite{Ehrenfeucht1984159} and Schrijver's proof of Vizing's theorem
\cite{schrijver}.

\begin{cor}[Vizing \cite{vizing}]
Every simple graph satisfies $\chi' \leq \Delta + 1$.
\end{cor}

In Sections \ref{VizingMultiSection} and \ref{FanSection} we generalize the multicoloring version of
Hall's theorem and use this to give a non-standard proof of the following
result from which all of the classical edge coloring results follow as well as various ``adjacency lemmas'' (see \cite{stiebitz} for the standard proof and how these consequences are
derived).

\begin{cor}\label{FanCor}
Let $G$ be a multigraph satisfying $\chi' \geq \Delta +
1$. For each critical edge $xy$ in $G$, there exists $X
\subseteq N(x)$ with $y \in X$ and $\card{X} \geq 2$ such that

\[\sum_{v \in X} \parens{d(v) + u(xv) + 1 - \chi'(G)} \geq 2. \]
\end{cor}

\section{The proof}

\begin{proof}[Proof of Theorem \ref{MainTheorem}]
First we prove necessity of the condition. Suppose we have $\W \subseteq
\S$ with $\card{\bigcup \W} < \card{\W} - \nu_t(\W)$.  We show that no matter what moves
$\F$ makes, we can maintain this invariant.  In particular, we
will always have $\card{\bigcup \W} < \card{\W}$ and hence $\W$ can never have a
transversal. Suppose $\F$ modifies $S \in \S$ by inserting $x$ and removing $y$
to get $S'$.  Put $\W_2 \DefinedAs \W \cup \set{S'} \smallsetminus \set{S}$.
First, if $S \not \in \W$, then $\B_t$ does nothing. So, we may assume $S \in
\W$.  Suppose $d_{\W}(x) > 0$.  Then $\card{\bigcup \W_2} \leq \card{\bigcup \W}$.  So, if $\nu_t(\W_2) = \nu_t(\W)$, the invariant is maintained.  Otherwise, $d_{\W_2}(x) - 1$ must be a multiple of $t+1$ and $d_{\W_2}(y)$ must not be a multiple of $t+1$; in particular, $d_{\W_2}(y) \neq d_{\W_2}(x) - 1$.  If $d_{\W_2}(y) < d_{\W_2}(x) - 1$ then $\B_t$ swaps $y$ in for one $x$ in $\W_2 \smallsetminus \set{S'}$ restoring the status quo.  Otherwise, we must have $d_{\W_2}(y) > d_{\W_2}(x) - 1$ and $\B_t$ swaps $x$ in for $y$ in $\min\set{t, d_{\W_2}(y) + 1 - d_{\W_2}(x)}$ sets of $\W_2$, decreasing $\nu_t(\W_2)$.

So, we may assume that $d_{\W}(x) = 0$ and hence $\card{\bigcup \W_2} = \card{\bigcup \W} + 1$. 
Now $\B_t$ swaps $x$ in for $y$ in $\min\set{t, d_{\W_2}(y)}$ sets of $\W_2$ to form $\W_3$.  
If $d_{\W_2}(y) \leq t$, then $d_{W_3}(y) = 0$ and we have $\card{\bigcup \W_3}
= \card{\bigcup \W}$ and the invariant is maintained.  Otherwise $\nu_t(\W_3) <
\nu_t(\W)$ and again the invariant is maintained.

Now we prove sufficiency.  Suppose the condition is not sufficient for $\F$ to
have a winning strategy and choose a counterexample $\S$ first minimizing
$\card{\S}$ and subject to that maximizing $\card{\bigcup \S}$.  

First, suppose $\card{\bigcup \S} \geq \card{\S}$.
Choose $\emptyset \neq C \subseteq \bigcup \S$ minimal such that $\card{\W_C} \leq
\card{C}$ where $\W_C \DefinedAs \setb{S}{\S}{C \cap S \neq \emptyset}$ (we
can make this choice because $\bigcup \S$ is such a subset).  Create a bipartite graph
with parts $C$ and $\W_C$ and an edge from $x \in C$ to $S \in \W_C$ iff $x \in
S$. If $\card{C} = 1$, then we clearly have a matching of $C$ into $\W_C$.  Otherwise,
by minimality of $C$, for every $\emptyset \neq D \subset C$ we have
$\card{\W_D} > \card{D}$ and hence $\card{C} = \card{\W_C}$; now applying Hall's
theorem (for bipartite graphs) gives a matching of $C$ into $\W_C$.  This matching gives
a transversal $\funcinj{f}{\W_C}{\bigcup \W_C}$ with $\im(f) = C$.  Put $\S'
\DefinedAs \S \smallsetminus \W_C$ and $P' \DefinedAs P \smallsetminus C$. Then the conditions of the theorem are satisfied with $\S'$ and $P'$ and if $\F$ plays on $\S'$ and $P'$, 
$\B_t$ cannot destroy the transversal of $\W_C$ that exists using elements of $C$, 
even though $\B_t$ may play on all of $\S$ (though still restricted to $P'$). Whence minimality of $\card{\S}$ gives a contradiction.

Therefore we must have $\card{\bigcup \S} < \card{\S}$ and hence $\nu_t(\S) \geq
1$.  Since $\card{P} \geq \card{\S}$, we have $y \in P$ with $d_{\S}(y) = 0$.  So, we may choose $x \in P$
with $d_{\S}(x) \geq t + 2$. Now $\F$ should swap $y$ in for $x$ in some $S \in
\S$ to form $\S_2$.  Since $d_{\S}(y) = 0$, we have $\card{\bigcup \S_2} >
\card{\bigcup \S}$.  We also have $d_{\S_2}(x) \geq t + 1$.  Now $\B_t$ moves
and creates $\S_3$. Then $d_{\S_3}(x) \geq d_{\S_2}(x) - t > 0$, so we have 
$\card{\bigcup \S_3} > \card{\bigcup \S}$.  Suppose our modifications changed
some $\W \in \S$ so it now violates the hypotheses, let $\W_3$ be $\W$
after the two player's moves.  Then more precisely, we mean $\card{\bigcup \W_3} < \card{\W_3} - \nu_t(\W_3)$. 
Since $d_{\S}(y) = 0$, $\card{\bigcup \W_3} \geq \card{\bigcup \W}$.  Thus
$\nu_t(\W_3) < \nu_t(\W)$ and hence $d_{\W_3}(x) < d_{\W}(x)$.  But then
$d_{\W_3}(y) > 0$ and $\card{\bigcup \W_3} > \card{\bigcup \W}$, a
contradiction.  Therefore, $\S_3$ satisfies the hypotheses of the theorem and
hence $\F$ can win by maximality of $\card{\bigcup \S}$.
\end{proof}

\section{Vizing's theorem}\label{VizingSection}
\noindent Vizing's theorem follows from a very special case of Theorem
\ref{MainTheorem}.

\begin{cor}\label{Simple}
If $\S = \set{S_1, \ldots, S_k}$ with $\card{S_k} \geq 1$ and $\card{S_i} \geq
2$ for all $i \in \irange{k-1}$, then $\F$ has a winning strategy against
$\B_1$.
\end{cor}

\begin{proof}
Let $\W \subseteq \S$. Then $\nu_1(\W) \geq \sum_{x
\in \bigcup \W} \frac{d_{\W}(x) - 2}{2} = \frac12 \sum_{S \in \W} \card{S} -
\card{\bigcup \W} \geq \frac12 (2\card{\W} - 1) -
\card{\bigcup \W}$.  Hence $\nu_1(\W) \geq \card{\W} - \card{\bigcup \W}$ as
desired.
\end{proof}

\begin{proof}[Proof of Vizing's theorem]
Suppose not and let $G$ be a counterexample minimizing $\card{G}$. Put $\Delta
\DefinedAs \Delta(G)$.  Pick $v \in V(G)$ with
degree $\Delta$, say $v_1, \ldots, v_\Delta$ are the neighbors of $v$ in $G$. 
By minimality of $\card{G}$, we have a $(\Delta + 1)$-edge-coloring of $G -
v$. Let $S_i$ be the colors not incident with $v_i$ in this coloring. 
Each $v_i$ has degree at most $\Delta - 1$ in $G - v$ and hence
$\card{S_i} \geq 2$. Also, if $a \in S_i$ and $b \not \in S_i$ we
may exchange colors on a maximum length path starting at $v_i$ and alternating between colors $b$ and $a$. 
This gives an $\F$ move followed by a $\B_1$ move.  Apply Corollary \ref{Simple}
to get a transversal of the $S_i$. Now we may complete the $(\Delta +
1)$-edge-coloring to all of $G$ by using the corresponding element of the
transversal on $vv_i$ for each $i \in \irange{\Delta}$.
\end{proof}

\section{The multicoloring version}\label{VizingMultiSection}
To deal with edge coloring multigraphs, we need to generalize our game slightly.
Instead of looking for a transversal, we will look for a system of
disjoint representatives. For $\func{\eta}{\S}{\IN^+}$ an
\emph{$\eta$-transversal} of $\S$ is a function $\func{f}{\S}{\P\parens{\bigcup
\S}}$ such that $f(S) \subseteq S$, $\card{f(S)} = \eta(S)$ for $S \in \S$ and
$f(A) \cap f(B) = \emptyset$ for different $A,B \in \S$.  By making $\eta(S)$
copies of each $S \in \S$ and applying Hall's theorem we get the following.

\begin{thm}
A set system $\S$ has an $\eta$-transversal iff $\card{\bigcup \W} \geq \sum_{W \in \W}
\eta(W)$ for each $\W \subseteq \S$.
\end{thm}

Call the game where $\F$ wins iff he creates an $\eta$-transversal \emph{the
$\eta$-game}.  We can use the same idea of making $\eta(S)$
copies of each $S \in \S$  to get a multicoloring version of Theorem \ref{MainTheorem}.  First, we need a lemma.

\begin{lem}\label{Spanner}
Let $G$ be a bipartite graph with parts $X$ and $Y$. If $\func{\eta}{X}{\IN^+}$ and $\card{N_G(X)} \geq \sum_{x \in X} \eta(x)$, then $G$ has a subgraph $H$ such that $d_H(x) = \eta(x)$ for each $x \in X \cap V(H)$, $d_H(y) = 1$ for each $y \in Y \cap V(H)$ and $N_G(Y \cap V(H)) \subseteq V(H)$.
\end{lem}
\begin{proof}
Create a bipartite graph $G'$ with parts $X'$ and $Y$ from $G$ by replacing each $x \in X$ with $\eta(x)$ identical copies of $x$.  By assumption, $\card{N_{G'}(X')} = \card{N_G(X)} \geq \card{X'}$.  Hence we can choose $\emptyset \neq C \subseteq N_{G'}(X')$ minimal such that $\card{N_{G'}(C)} \leq \card{C}$.  If $\card{C} = 1$, then we clearly have a matching $M$ of $C$ into $N_{G'}(C)$.  Otherwise, by minimality of $C$, for every $\emptyset \neq D \subset C$, we have $\card{N_{G'}(D)} > \card{D}$ and hence $\card{N_{G'}(C)} = \card{C}$; now applying Hall's theorem (for bipartite graphs) gives a matching $M$ of $C$ into $N_{G'}(C)$. Since all copies of $x \in X$ have the same neighborhood, we see that a copy of $x \in X$ is in $N_{G'}(C)$ iff all copies of $x$ are.  For $x \in X$, let $O_x$ be the set of copies of $x$ in $X'$ and let $H$ be the graph with vertex set $\setbs{O_x}{O_x \subseteq N_{G'}(C)} \cup C$ and edge set $\setbs{O_xy}{zy \in M \text{ for some $z \in O_x$}}$.  Then $H$ is (isomorphic to) a subgraph of $G$ with the desired properties.
\end{proof}

\begin{thm}\label{MainTheoremMulti}
In a set system $\S$ with $\bigcup \S \subseteq P$ and $\card{P} \geq \sum_{S \in \S}
\eta(S)$, $\F$ has a winning strategy against $\B_t$ in the $\eta$-game iff $\card{\bigcup
\W} \geq \sum_{W \in \W}
\eta(W) - \nu_t(\W)$ for each $\W \subseteq \S$.
\end{thm}
\begin{proof}
First we prove necessity of the condition. We note that the proof of necessity is identical to that in Theorem \ref{MainTheorem} aside from changing the invariant we are maintaining.  Suppose we have $\W \subseteq
\S$ with $\card{\bigcup\W} < \sum_{W \in \W} \eta(W) - \nu_t(\W)$.  We show that no matter what moves
$\F$ makes, we can maintain this invariant.  In particular, we
will always have $\card{\bigcup\W} < \sum_{W \in \W} \eta(W)$ and hence $\W$ can never have an
$\eta$-transversal. Suppose $\F$ modifies $S \in \S$ by inserting $x$ and removing $y$
to get $S'$.  Put $\W_2 \DefinedAs \W \cup \set{S'} \smallsetminus \set{S}$.
First, if $S \not \in \W$, then $\B_t$ does nothing. So, we may assume $S \in
\W$.  Suppose $d_{\W}(x) > 0$.  Then $\card{\bigcup \W_2} \leq \card{\bigcup \W}$.  So, if $\nu_t(\W_2) = \nu_t(\W)$, the invariant is maintained.  Otherwise, $d_{\W_2}(x) - 1$ must be a multiple of $t+1$ and $d_{\W_2}(y)$ must not be a multiple of $t+1$; in particular, $d_{\W_2}(y) \neq d_{\W_2}(x) - 1$.  If $d_{\W_2}(y) < d_{\W_2}(x) - 1$ then $\B_t$ swaps $y$ in for one $x$ in $\W_2 \smallsetminus \set{S'}$ restoring the status quo.  Otherwise, we must have $d_{\W_2}(y) > d_{\W_2}(x) - 1$ and $\B_t$ swaps $x$ in for $y$ in $\min\set{t, d_{\W_2}(y) + 1 - d_{\W_2}(x)}$ sets of $\W_2$, decreasing $\nu_t(\W_2)$.

So, we may assume that $d_{\W}(x) = 0$ and hence $\card{\bigcup \W_2} = \card{\bigcup \W} + 1$. 
Now $\B_t$ swaps $x$ in for $y$ in $\min\set{t, d_{\W_2}(y)}$ sets of $\W_2$ to form $\W_3$.  
If $d_{\W_2}(y) \leq t$, then $d_{W_3}(y) = 0$ and we have $\card{\bigcup \W_3}
= \card{\bigcup \W}$ and the invariant is maintained.  Otherwise $\nu_t(\W_3) <
\nu_t(\W)$ and again the invariant is maintained.

Now we prove sufficiency.  Suppose the condition is not sufficient for $\F$ to
have a winning strategy and choose a counterexample $\S$ first minimizing
$\card{\S}$ and subject to that maximizing $\card{\bigcup \S}$. 

First, suppose $\card{\bigcup \S} \geq \sum_{S \in \S} \eta(S)$.  Let $G$ be the bipartite graph with parts $\S$ and $\bigcup \S$ and an edge from $S \in \S$ to $y \in \bigcup \S$ iff $y \in S$.  Apply Lemma \ref{Spanner} to get a subgraph $H$ of $G$ such that $d_H(S) = \eta(S)$ for each $S \in \S \cap V(H)$, $d_H(y) = 1$ for each $y \in \bigcup \S \cap V(H)$ and $N_G(\bigcup \S \cap V(H)) \subseteq V(H)$.  Then $\func{f}{\S \cap V(H)}{\P\parens{\bigcup \S}}$ defined by $f(S) \DefinedAs N_H(S)$ is an $\eta$-transversal of $\S \cap V(H)$ with $\bigcup \im(f) = \bigcup \S \cap V(H)$. Put $\S' \DefinedAs \S \smallsetminus V(H)$ and $P' \DefinedAs P \smallsetminus V(H)$. Then the conditions of the theorem are satisfied with $\S'$ and $P'$ and if $\F$ plays on $\S'$ and $P'$, $\B_t$ cannot destroy the transversal of $\S \cap V(H)$ that exists using elements of $\bigcup \S \cap V(H)$, even though $\B_t$ may play on all of $\S$ (though still restricted to $P'$). Whence minimality of $\card{\S}$ gives a contradiction.

Therefore we must have $\card{\bigcup \S} < \sum_{S \in \S} \eta(S)$ and hence $\nu_t(\S) \geq
1$. Since $\card{P} \geq \sum_{S \in \S} \eta(S)$, we have $y \in P$ with $d_{\S}(y) = 0$.  So, we may choose $x \in P$
with $d_{\S}(x) \geq t + 2$. Now $\F$ should swap $y$ in for $x$ in some $S \in
\S$ to form $\S_2$.  Since $d_{\S}(y) = 0$, we have $\card{\bigcup \S_2} >
\card{\bigcup \S}$.  We also have $d_{\S_2}(x) \geq t + 1$.  Now $\B_t$ moves
and creates $\S_3$. Then $d_{\S_3}(x) \geq d_{\S_2}(x) - t > 0$, so we have 
$\card{\bigcup \S_3} > \card{\bigcup \S}$.  Suppose our modifications changed
some $\W \in \S$ so it now violates the hypotheses, let $\W_3$ be $\W$
after the two player's moves.  Then more precisely, we mean $\card{\bigcup \W_3} < \sum_{W \in \W_3}
\eta(W) - \nu_t(\W_3)$. 
Since $d_{\S}(y) = 0$, $\card{\bigcup \W_3} \geq \card{\bigcup \W}$.  Thus
$\nu_t(\W_3) < \nu_t(\W)$ and hence $d_{\W_3}(x) < d_{\W}(x)$.  But then
$d_{\W_3}(y) > 0$ and $\card{\bigcup \W_3} > \card{\bigcup \W}$, a
contradiction.  Therefore, $\S_3$ satisfies the hypotheses of the theorem and
hence $\F$ can win by maximality of $\card{\bigcup \S}$.
\end{proof}

\section{The fan equation}\label{FanSection}

\begin{proof}[Proof of Corollary \ref{FanCor}]
Put $k \DefinedAs \chi'(G) - 1$.  Consider a $k$-edge-coloring $\pi$ of $G - xy$.  For $v \in N(x)$, let
$M_v \subseteq \irange{k}$ be those colors not incident to $v$ under $\pi$ and let $D_v$ be the colors on the edges from $x$ to $v$.  Then the $D_v$ are
pairwise disjoint, $\card{D_v} = \mu(xv)$ for $v \in N(x) - \set{y}$ and
$\card{D_y} = \mu(xy) - 1$.  For $v \in N(x)$, put $S_v \DefinedAs M_v \cup
D_v$. Then $\card{S_v} = \card{D_v} + \card{M_v} = k + \mu(xv) - d(v)$.  

Now we translate the problem into our game.  Put $\eta(S_v) \DefinedAs \mu(xv)$. 
If $v \in N(x)$ and $a \in S_v$ and $b \not \in S_v$ we may exchange colors on a maximum length path in $G-x$ starting at $v$ and alternating between colors $b$ and $a$. 
This gives an $\F$ move followed by a $\B_1$ move in the $\eta$-game with sets
$\S_{N(x)}$ where $\S_X \DefinedAs \setbs{S_v}{v \in X}$ for $X \subseteq N(x)$. Plainly, if $\F$ has a winning
strategy in this $\eta$-game against $\B_1$, then we can extend the $k$-edge-coloring to all of $G$ giving a contradiction.  

Therefore, by Theorem \ref{MainTheoremMulti}, we must have $X \subseteq N(x)$
with $\card{\bigcup_{v \in X} S_v} < \sum_{v \in X} \eta(S_v) - \nu_1(\S_X) = \sum_{v
\in X} \mu(xv) - \nu_1(\S_X)$.  Since the $D_v$ are pairwise disjoint, we have
$\card{\bigcup_{v \in X} S_v} \geq -1 + \sum_{v \in X} \mu(xv)$ with equality
only if $y \in X$.  Hence $y \in X$ and $\nu_1(\S_X) = 0$.  Since $\chi' \geq
\Delta + 1$, we have $\card{S_y} = k + \mu(xy) - d(y) \geq \mu(xy)$ and hence
$\card{X} \geq 2$.  Since $\nu_1(\S_X) = 0$, each color is in at most two
elements of $\S_X$.  Therefore $\sum_{v \in X} \parens{k + \mu(xv) - d(v)} =
\sum_{v \in X} \card{S_v} \leq 2 \card{\bigcup_{v \in X} S_v} \leq -2 + 2\sum_{v \in X} \mu(xv)$.  The corollary follows.
\end{proof}

\section{Acknowledgements}
\noindent Thanks to the anonymous referees for helping to improve the readability of the paper.

\bibliographystyle{amsplain}
\bibliography{hall}
\end{document}